\documentclass[12 pt,a4paper]{amsart}
\usepackage{amsfonts}
\usepackage{amscd,amsmath}
\usepackage{amssymb}
\usepackage{setspace}
\usepackage{fullpage}
\usepackage{amsfonts}
\usepackage{amscd,amsmath}
\usepackage{amssymb}
\usepackage{tikz-cd}
\usepackage{mathtools}
\newtheorem{theorem}{Theorem}[section]

\newtheorem{corollary}[theorem]{Corollary}
\newtheorem{proposition}[theorem]{Proposition}
\newtheorem{definition}[theorem]{Definition}

\newtheorem*{TI*}{Topological Interpretation}

\theoremstyle{remark}
\newtheorem{remark}{Remark}[section]
\newtheorem*{rem*}{Remark}

\newcommand{\comment}[1]{}
\numberwithin{equation}{section}
\DeclareMathOperator{\Um}{Um}
\DeclareMathOperator{\fr}{first}
\DeclareMathOperator{\co}{column}
\DeclareMathOperator{\of}{of}

\begin{document}
\title{Algebraic cohomotopy groups,  algebraic fundamental groups and stably free modules}
\author{Raja Sridharan $^1$ and Sumit Kumar Upadhyay $^2$\vspace{.4cm}\\
{$^{1}$School Of Mathematics,\\ Tata Institute of fundamental Research,\\Colaba,  Mumbai, India \vspace{.3cm}\\ $^{2}$Department of Applied Sciences,\\ Indian Institute of Information Technology Allahabad, \\Prayagraj, U. P., India\\}}
\thanks{$^1$sraja@math.tifr.res.in, $^2$upadhyaysumit365@gmail.com}
%\thanks{E-mail: upadhyaysumit365@gmail.com}.\\

\begin{abstract}
In this article, we prove the algebraic counterpart of the topological results $H^1(S^1, \mathbb{Z}) \cong \mathbb{Z}$ and $H^1(S^2, \mathbb{Z}) \cong \{0\}$. We also see that a non-trivial element of the algebraic cohomotopy groups of certain rings associated with some known topological spaces provides examples of non-free   stably free module of rank two over those rings.
\end{abstract}
\maketitle
\section{Introduction}
A generalization of the concept of a vector space over a field is the concept of a module over a ring. Although every vector space has a basis, modules over rings do not in general.  Free modules are those that have a basis. A projective module is a direct summand of a free module. In fact,  the class of projective modules enlarges the class of free modules over a ring by keeping some of the main properties of free modules intact. In the realm of mathematics, the freeness of finitely generated projective modules over a ring therefore emerged as an interesting research problem. In 1955, Serre \cite{Serre} asked the question whether projective modules over polynomial rings over fields are free. Although Serre was not sure whether this question had an affirmative answer, the question became known to the mathematical world as ``Serre's Conjecture". Serre's Conjecture was motivated by the result that over  Euclidean space, topological vector bundles are trivial. In 1957/1958, Serre \cite{Serre1} made progress and proved that every finitely generated projective module over a polynomial ring over a field is stably free. In 1976, D. Quillen and A. Suslin (\cite{DQ, SU}) proved independently that Serre's conjecture is true.  It is well known that the study of stably free modules over a ring is equivalent to the study of unimodular rows over the ring. There are a lot of theories that have been developed by many mathematicians to decide whether projective modules corresponding to unimodular rows are free. 

 Now, suppose $A$ is the coordinate ring of a real affine variety $X$ with real points $X(\mathbb{R})$. Then an element of $A$ gives a continuous map from $X(\mathbb{R})$ to $\mathbb{R}$. An element $(a_1, a_2, \ldots, a_n)\in A^n$ is unimodular if there exist $b_1, b_2, \ldots, b_n\in A$ such that $a_1b_1+a_2b_2+\cdots a_nb_n =1$ which means that $a_1, a_2, \ldots, a_n$ have no common zero in $X(\mathbb{R})$. Thus, every unimodular row of length $n$ over $A$ gives a continuous map from $X(\mathbb{R})$ to $\mathbb{R}^n \setminus \{(0, 0, \cdots, 0)\}$. So, one can say that set of all unimodular rows over $A$ is the algebraic counterpart of the set of all continuous maps from $X(\mathbb{R})$ to $\mathbb{R}^n \setminus \{(0, 0, \cdots, 0)\}$ (for details see \cite{SSR}). Using this as motivation, one can ask for algebraic analogues of the groups $H^0(X, \mathbb{Z})$ and $H^1(X, \mathbb{Z})$ of a topological space $X$. In turns out, the correct algebraic analogues of these groups were given by Nori. Nori's definition was based on earlier work by Krusemeyer \cite{MK}.  In 2020, Raja Sridharan et. al. \cite{RSS} discussed these algebraic counterpart of the groups  $H^0(X, \mathbb{Z})$ and $H^1(X, \mathbb{Z})$ and we termed these groups as algebraic fundamental groups and algebraic cohomotopy groups, respectively. 
 
It is known that the Mayer Vietoris sequence (see \cite{CTC}) is one of the primary techniques used to determine the group $H^1(X, \mathbb{Z})$ of a topological space $X$. In \cite{RSS}, the authors proved an algebraic counterpart of the Mayer Vietoris sequence associated with the fibre product diagram $\begin{tikzcd}
A \arrow{r}{} \arrow[swap]{d}{} & A_a \arrow{d}{} \\
A_b\arrow{r}{} & A_{ab}
\end{tikzcd}$ where $A$ is ring and the ideal generated by $a$ and $b$ is $A$. This fibre product diagram corresponds to the fact that a topological space $X$ is the union of two open sets as a condition needed to obtain the topological Mayer Vietoris sequence. Numerous characteristics of algebraic cohomotopy groups and algebraic fundamental groups have been demonstrated in \cite{RSS1} and \cite{RSS2}, respectively. 

Now, we provide some heuristic motivation following the insights of Krusemeyer and Nori. 
Consider the field of complex numbers $\mathbb{C}$ and the exponential map $\exp : \mathbb{C} \rightarrow \mathbb{C}^*$, where $\mathbb{C}^* = \mathbb{C} \setminus \{0\}$. We then have the following short exact sequence of groups
$$0 \rightarrow \mathbb{Z} \rightarrow \mathbb{C} \overset{\exp}\rightarrow \mathbb{C}^* \rightarrow 0 \hspace{2cm} (1)$$
where one can consider $\mathbb{C}$ as universal covering space of $\mathbb{C}^*$. 

%Krusemeyer algebraized the above exponential sequence. By dualizing sequence $(1)$, we discuss this algebraization. 

Let $X$ and $Y$ be two topological spaces. Consider the set $\text{Cont}(X, Y)$ of all continuous maps from $X$ to $Y$. It is dualizing the above sequence $(1)$, we have the following short exact sequence of sheaves

$$1 \rightarrow \text{Cont}(X, \mathbb{Z}) \rightarrow \text{Cont}(X, \mathbb{C}) \overset{\exp}\rightarrow \text{Cont}(X, \mathbb{C}^* )\rightarrow 1 \hspace{2cm} (2)$$

Note that  $\exp:\text{Cont}(X, \mathbb{C}) \rightarrow \text{Cont}(X, \mathbb{C}^* )$ need not be surjective since the identity map from $\mathbb{C}^*$ to $\mathbb{C}^*$ does not admit a global logarithm. However exponential map is locally surjective, yielding a exact sequence of sheaves.  

We now discuss the algebraisation of sequence (1) due to Krusemeyer.  To do this, one replaces $\mathbb{C}^*$ by $SL_2 (\mathbb{R})$ which is homotopy equivalent to $\mathbb{C}^*$.  Then we have the following exact sequence (for details see \cite{RSS2}) 

$$0 \rightarrow \pi_1(SL_2(\mathbb{R}), I_2) \rightarrow E \overset{p}\rightarrow SL_2 (\mathbb{R}) \rightarrow 0\hspace{2cm}  (2)$$
 where  $E = \{[\alpha(T)] \mid  \alpha(T) \in SL_2(\mathbb{R}[T]) ~\text{with}~ \alpha(0) = I_2 \}$ is the path space and $[\alpha(T)]$ denotes equivalence  class of $\alpha(T)$ (we say that  $\alpha(T) \sim \beta(T)$ if $\alpha(1) = \beta(1)$ and there exists $\gamma(T, S) \in SL_2 (\mathbb{R}[T, S])$ such that 
\begin{center}
$\gamma(T, 0) = \alpha(T)$, $\gamma(T, 1) = \beta(T)$\\
$\gamma(0, S) = I_2$, $\gamma(1, S) = \alpha(1) =\beta(1)$)
\end{center}
 In this case, $E$ is the algebraic universal covering space of $SL_2(\mathbb{R})$ and $p$ is the map sending $[\alpha(T)]$ to $\alpha(1)$.
 
 To dualize the above sequence note that if $X$ is a topological space, then $\text{Cont}(X, SL_2(\mathbb{R}))$ is equal to $SL_2(A)$, where $A = \text{Cont}(X, \mathbb{R})$. Krusemeyer dualizes sequence $(2)$ to obtain the following exact sequence of sheaves
 
 $$1 \rightarrow \pi_1(SL_2(A, I_2)) \rightarrow G'(A) \overset{p}\rightarrow SL_2 (A) \rightarrow 1\hspace{2cm}  (3)$$. We note that $p$ is not necessarily surjective. In sequence $(3)$, $p$ is the map sending $[\alpha(T)]$ to $\alpha(1)$ and   $G'(A) = \{[\alpha(T)] \mid  \alpha(T) \in SL_2(A[T]) ~\text{with}~ \alpha(0) = I_2 \}$ where $[\alpha(T)]$ denotes the equivalence class of $\alpha(T) \in SL_2(A[T])$ with $\alpha(0)= I_2$ and the equivalence relation $\sim$ is defined as follows:

$\alpha(T) \sim \beta(T)$ if $\alpha(1) = \beta(1)$ and there exists $\gamma(T, S) \in SL_2 (A[T, S])$ such that 
\begin{center}
$\gamma(T, 0) = \alpha(T)$, $\gamma(T, 1) = \beta(T)$\\
$\gamma(0, S) = I_2$, $\gamma(1, S) = \alpha(1) =\beta(1)$.
\end{center} 

Note that we have following long exact cohomology sequence associated to the exponential sequence $(2)$
 
 $0 \to H^0(X, \mathbb{Z})\to H^0(X, \mathbb{C}) \to H^0(X, \mathbb{C}^*)\to H^1(X, \mathbb{Z})\to H^1(X, \mathbb{C}) \to H^1(X, \mathbb{C}^*)$ $\to H^2(X, \mathbb{Z})\to H^2(X, \mathbb{C}) \to H^2(X, \mathbb{C}^*)\to \cdots$
 
Nori's idea is that the algebraic analog of the above topological cohomology sequence should be a long exact C\^ech cohomology sequence associated with a short exact sequence of sheaves $(3)$, where the topology on $\text{Spec} A$ is the Zariski topology. We shall consider some particular consequences of Nori's algebraisation. We would be grateful for any references to the literature where similar ideas occur.
 
\begin{enumerate}
\item The group $\pi_1(SL_2(A))$ is the algebraic analogue of $H^0(X, \mathbb{Z})$ and is related to connectedness (see \cite{RSS2}).
\item Since $H^1(X, \mathbb{C})$ is zero, it follows that $H^1(X, \mathbb{Z})$ is equal to $\text{Cont}(X, \mathbb{C}^*)$ modulo image of exponential of $\text{Cont}(X, \mathbb{C})$ that is those elements of $\text{Cont}(X, \mathbb{C}^*)$ which are homotopic to a constant. Similarly, $H^1(X, \mathbb{Z})$ should be similar to $\text{Cont}(X, SL_2(\mathbb{R}))$ modulo those maps which are homotopic to constant that is, the algebraic analogue of $H^1(X, \mathbb{Z})$ is the group $\Gamma (A) = SL_2(A)/QL_2(A)$ (see \cite{RSS1}), where $QL_2(A)  = \{\alpha \in SL_2(A)\mid \alpha ~\text{is connected to the identity}\}$ (\cite{SM}). This group is generalization of the group $SL_n(A)/E_n(A)$ ($n\geq 3$) which has been considered in the literature. 
\item If $X = U \cup V$ where $U$ and $V$ are open contractible sets, then $H^1(X, \mathbb{Z})$ is equal to $\text{Cont}(X, \mathbb{Z})$ modulo those continuous functions which can be written as $\beta - \gamma$ where $\beta: U \to \mathbb{Z}$ and $\gamma: V \to \mathbb{Z}$ are continuous.

Suppose $A$ is ring and $U = D(f)$ and $V = D(g)$, where $\langle f, g\rangle = A$. Then using (1), we see that $H^1(X, \mathbb{Z})$ is similar to group $\frac{\pi_1(SL_2(A_{fg})}{\pi_1(SL_2(A_{f}))\times \pi_1(SL_2(A_{g}))}.$ The algebraic analogue of Mayer-Vietoris sequence for two open sets is studied in \cite{RSS}.
\end{enumerate}
In this paper, we demonstrate a certain algebraic generalization of Mayer-Vietoris sequence for closed sets (see \cite{CTC}). This sequence allows us to verify the algebraic counterparts of the facts that $H^1(S^1, \mathbb{Z}) \cong \mathbb{Z}$ and $H^1(S^2, \mathbb{Z}) \cong \{0\}$. As an application, we also see that a non-trivial element of the algebraic cohomotopy groups of certain rings associated with some known topological spaces provides examples of  non-free stably free modules of rank two over those rings. 

Throughout the paper, we say a unimodular row $(a_1, a_2, \ldots, a_n) \in A^n$ is completable if there exists a matrix $\alpha\in SL_n(A)$ having $(a_1, a_2, \ldots, a_n)$ as  the first column. Now, we recall known definitions of algebraic fundamental groups and algebraic cohomotopy groups (see for example \cite{RSS}). 
\begin{definition}
Let $A$ be a ring and $L$ be the set of all loops in $SL_{2}(A)$ starting and ending at the identity matrix $I_2$, that is, $L =\{\alpha(T)\in SL_{2}(A[T])\mid \alpha(0) = \alpha(1) = I_2\}$. We say that two loops $\alpha(T), \beta(T) \in L$ are equivalent (that is, written as $\alpha(T)\sim \beta(T)$) if there exists $\gamma(T, S)\in SL_{2}(A[T, S])$ such that $\gamma(T, 0) = \alpha(T), \gamma(T, 1) = \beta(T)$ and $\gamma(0,S) = \gamma(1, S) = I_2$. We call $\gamma(T, S)$ as a homotopy between $\alpha(T)$ and $\beta(T)$.

Let $\pi_1 (SL_{2}(A))$ be the set of all equivalence classes of loops based on $I_2$. It forms an abelian group with respect to  the binary operation `$*$' defined by $[\alpha(T)]*[\beta(T)] = [\alpha(T)\beta(T)]$. We call $\pi_1 (SL_{2}(A))$ the algebraic fundamental group of $A$. 
\end{definition}

\begin{definition}
We say that two unimodular rows $(a, b),~ (c, d)$ over $A$ are equivalent, written as $(a, b)\sim (c, d)$, if one (and hence both) of the following equivalent conditions holds:
\begin{enumerate}
\item there exists $(f_{1} (T), f_{2}(T)) \in \Um_2 (A[T])$ such that $(f_{1} (0), f_{2}(0)) = (a, b)$ and $(f_{1} (1), f_{2}(1)) = (c, d)$.
\item there exists a matrix $\alpha \in SL_2 (A)$ which is connected to the identity matrix (that is, there exists a matrix $\beta(T) \in SL_2 (A[T])$ such that $\beta(0) = I_2$ and $\beta(1) = \alpha$) such that
$\alpha \begin{pmatrix}
a \\ b
\end{pmatrix} = \begin{pmatrix}
c \\ d
\end{pmatrix}$.
\end{enumerate}

It is not hard to check that the relation $\sim$ is an equivalence relation. We denote  the equivalence class of $(a,b)$ by $[a,b]$. Let $\Gamma(A)$ be the set of all equivalence classes of unimodular rows given by the equivalence relation $\sim$ as above. Define a product $*$ in $\Gamma(A)$ as follows:

Let $(a, b),~ (c, d) \in \Um_2 (A)$. Suppose $\sigma = \begin{pmatrix}
a & e \\ b & f
\end{pmatrix} $ and $\tau = \begin{pmatrix}
c & g \\ d & h
\end{pmatrix}$ are completions of $(a, b)$ and $(c, d)$ in $SL_2(A)$, respectively. We define the product of two elements $[a,b], [c,d]\in \Gamma(A)$ as follows: $$[a, b] * [c, d] = [\fr\;\co\;\of~ \sigma\tau] = [ac+de, bc+df].$$

Then $(\Gamma(A), *)$ is a group and the identity element of the group $\Gamma(A)$ is $[1, 0]$ (for details see \cite{RSS}). We called it as the algebraic cohomotopy group of $A$. 
\end{definition}
\begin{remark}
\begin{enumerate}
\item Let $f: A \to B$ be a ring homomorphism. Then we have a group homomorphism $\Gamma(f): \Gamma(A) \to \Gamma(B)$ defined as $\Gamma(f)([a, b])= [f(a), f(b)]$.
\item In fact, $\Gamma(A) = \frac{SL_2(A)}{QL_2(A)}$ (for $QL_2(A)$, one may see \cite{SM}). In this case, the identity element of the group $\Gamma(A)$ is the identity matrix $I_2.$
\end{enumerate}

\end{remark}

\section{\textbf{An additional algebraic counterpart of the Mayer-Vietoris sequence}}
Let $A$ be a commutative ring with identity. Consider the pull back of the diagram 
\begin{equation}\label{1}
\begin{tikzcd}
B \arrow{r}{i_1} \arrow[swap]{d}{i_2} & A[X] \arrow{d}{\delta} \\
A\arrow{r}{\Delta} & A\oplus A
\end{tikzcd}
\end{equation} where  $\Delta: A \to A\oplus A$ and $\delta: A[X] \to A\oplus A$ are the ring homomorphisms defined as $\Delta(a)= (a, a)$, and  $\delta(f(X))= (f(0), f(1))$, respectively. Clearly, $B = \{(f(X), a) \mid \delta(f(X)) = \Delta(a)\} = \{(f(X), a) \mid f(0) = f(1) = a\}$. Hence, $SL_2(B) = \{(\alpha(X), \beta)\in SL_2(A[X])\times SL_2(A) \mid \alpha(0) = \alpha(1) = \beta\}$.

\begin{theorem}\label{direct sum}
There exists a group homomorphism from $\pi_1(SL_2(A\oplus A))$ to $\Gamma(B)$ with the kernel $\pi_1(SL_2(A))$.
\end{theorem}
\begin{proof}
Let $([\alpha(T)], [\beta(T)])\in \pi_1(SL_2(A\oplus A))$ and  $M_{\alpha}^{\beta}(X, T) = \alpha((1-X)T)\beta(XT) \in SL_2(A[X, T]$. Then, $$M_{\alpha}^{\beta}(0, T) = \alpha(T), M_{\alpha}^{\beta}(1, T) = \beta(T)$$
$$M_{\alpha}^{\beta}(X, 0) = I_2, M_{\alpha}^{\beta}(0, 1) =M_{\alpha}^{\beta}(1, 1) = I_2.$$
Therefore, $(M_{\alpha}^{\beta}(X, 1),  I_2) \in \Gamma(B)$. Thus, we can get a map $\chi: \pi_1(SL_2(A\oplus A)) \to \Gamma(B)$ defined by $$\chi([\alpha(T)], [\beta(T)]) = (M_{\alpha}^{\beta}(X, 1),  I_2).$$ 

\noindent\textbf{Claim 1: $\chi$ is independent on the choice of  $M_{\alpha}^{\beta}(X, T)$ and is well defined.} 

Suppose $M_{\alpha}^{\beta}(X, T), M_{\alpha}^{\beta}(X, T)'\in SL_2(A[X, T]$ such that $$M_{\alpha}^{\beta}(0, T) =M_{\alpha}^{\beta}(0, T)' = \alpha(T), M_{\alpha}^{\beta}(1, T) =M_{\alpha}^{\beta}(1, T)' = \beta(T)$$
$$M_{\alpha}^{\beta}(X, 0) =M_{\alpha}^{\beta}(X, 0)' = I_2 , M_{\alpha}^{\beta}(0, 1) = M_{\alpha}^{\beta}(1, 1) =M_{\alpha}^{\beta}(0, 1)' =M_{\alpha}^{\beta}(1, 1)' = I_2.$$
Now, consider $\theta(X, T) = M_{\alpha}^{\beta}(X, T)' {M_{\alpha}^{\beta}(X, T)}^{-1}$. Then $(\theta(X, T), I_2) \in SL_2(B[T])$ (as $\theta(0, T) =  \theta(1, T) = I_2$). Also, $$(\theta(X, 0), I_2)  = (I_2, I_2)$$ and $$(\theta(X, 1), I_2) (M_{\alpha}^{\beta}(X, 1),  I_2) = (M_{\alpha}^{\beta}(X, 1)',  I_2).$$ This shows that $\chi$ is independent on the choice of  $M_{\alpha}^{\beta}(X, T)$.

Suppose $([\alpha(T)], [\beta(T)]) = ([\alpha'(T)], [\beta'(T)])\in \pi_1(SL_2(A\oplus A))$. Then there exists $(\gamma_1(T, S), \gamma_2(T, S)) \in SL_2((A\oplus A)[T, S])$ such that 
$$(\gamma_1(T, 0), \gamma_2(T, 0)) = (\alpha(T), \beta(T)), (\gamma_1(T, 1), \gamma_2(T, 1)) = (\alpha'(T), \beta'(T))$$ and 
$$(\gamma_1(0, S), \gamma_2(0, S)) = (\gamma_1(1, S), \gamma_2(1, S)) = (I_2, I_2).$$

Consider, $M(X, T, S) = \gamma_1((1-X)T, S) \gamma_2(XT, S)$. Then it is easy to see that $\chi([\alpha(T)], [\beta(T)]) = (M(X, 1, 0),  I_2)$ and $\chi([\alpha'(T)], [\beta'(T)]) = (M(X, 1, 1),  I_2)$. 

Take $\theta(X, S) = M(X, 1, S)M(X, 1, 0)^{-1}$. Then 
$$(\theta(0, S), I_2) = (\theta(1, S), I_2) = (I_2, I_2)$$
$$(\theta(X, 0), I_2) = (I_2, I_2) ~\text{and}~ (\theta(X, 1), I_2)(M(X, 1, 0),  I_2) =(M(X, 1, 1),  I_2).$$ This shows that $(M(X, 1, 0),  I_2) =(M(X, 1, 1),  I_2)$ in $\Gamma(B)$, that is, $\chi$ is well defined. 

\noindent\textbf{Claim 2: $\chi$ is a group homomorphism.} 

Suppose $\chi([\alpha(T)], [\beta(T)]) = (M_{\alpha}^{\beta}(X, 1),  I_2)$ and $\chi([\alpha'(T)], [\beta'(T)]) = (M_{\alpha'}^{\beta'}(X, 1),  I_2)$

Consider $M(X, T) = M_{\alpha}^{\beta}(X, T)M_{\alpha'}^{\beta'}(X, T)$. Then we have $$M(0, T) = \alpha(T)\alpha'(T), M(1, T) = \beta(T)\beta'(T)$$
$$M(X, 0) = I_2, M(0, 1) =M(1, 1) = I_2.$$ Hence, $\chi([\alpha(T)\alpha'(T)], [\beta(T)\beta'(T)]) = (M(X, 1),  I_2)= (M_{\alpha}^{\beta}(X, 1),  I_2)(M_{\alpha'}^{\beta'}(X, 1),  I_2)$. Thus
$\chi([\alpha(T)], [\beta(T)])\chi([\alpha'(T)], [\beta'(T)]) = \chi\big(([\alpha(T)], [\beta(T)])([\alpha'(T)], [\beta'(T)])\big)$.

\noindent\textbf{Claim 3: $\ker(\chi) \cong \pi_1(SL_2(A))$.} 

To prove the claim, it is enough to show that $\ker(\chi) =  \{([\alpha(T)], [\alpha(T)]) ~\mid~ \alpha(T) \in \pi_1(SL_2(A))\}$.

It is clear that $([\alpha(T)], [\alpha(T)]) \in \ker(\chi)$. Conversely, suppose that $([\alpha(T)], [\beta(T)])\in \ker(\chi)$. Then $(M_{\alpha}^{\beta}(X, 1),  I_2) = (I_2, I_2)$ in $\Gamma(B)$, where $M_{\alpha}^{\beta}(0, T) = \alpha(T), M_{\alpha}^{\beta}(1, T) = \beta(T),$ $M_{\alpha}^{\beta}(X, 0) = I_2, M_{\alpha}^{\beta}(0, 1) =M_{\alpha}^{\beta}(1, 1) = I_2.$

Since $(M_{\alpha}^{\beta}(X, 1),  I_2) = (I_2, I_2)$ in $\Gamma(B)$, there exists $(\theta(X, T), I_2)\in SL_2(B[X, T])$ such that $$(\theta(X, 0), I_2)= (I_2, I_2)$$ and $$(\theta(X, 1), I_2)(M_{\alpha}^{\beta}(X, 1),  I_2) = (I_2, I_2).$$ Since $(\theta(X, T), I_2)\in SL_2(B[X, T])$, $\theta(0, T)= \theta(1, T)= I_2$. 

Take $\gamma(X, T) = \theta(X, T)M_{\alpha}^{\beta}(X, T)$. Then 
$\gamma(0, T) = \alpha(T), \gamma(1, T) = \beta(T)$ and $\gamma(X, 0)=  \gamma(X, 1) = I_2$. Thus $[\alpha(T)] =  [\beta(T)]$. Therefore, $\ker(\chi) \cong \pi_1(SL_2(A))$.
\end{proof}

\begin{corollary}\label{trivial}
Let $A$ be a ring such that $QL_2(A) = SL_2(A)$, that is, $\Gamma(A)$ is trivial. Then $\chi$ is surjective. Moreover, $\pi_1(SL_2(A))\cong \Gamma(B).$
\end{corollary}
\begin{proof}
Let $(\alpha(X), \beta)\in \Gamma(B)$. Then $\beta \in SL_2(A)=QL_2(A)$. So, there exists a $\sigma(T) \in SL_2(A[T])$ such that $\sigma(0) = I_2$ and $\sigma(1) = \beta$. Consider $(\sigma(T)^{-1}, \sigma(T)^{-1})$. Clearly, $(\sigma(T)^{-1}, \sigma(T)^{-1}) \in SL_2(B[T])$. Also $(\sigma(0)^{-1}, \sigma(0)^{-1}) = (I_2, I_2)$  and $(\sigma(1)^{-1},\\ \sigma(1)^{-1})(\alpha(X), \beta) = (\beta^{-1}\alpha(X), I_2)$. This shows that $(\alpha(X), \beta) = (\beta^{-1}\alpha(X), I_2)$ in $\Gamma(B)$.

Let $(\alpha(X), I_2)\in \Gamma(B)$. Then $\alpha(X) \in \Gamma(A[X])$. Since $\Gamma(A) = I_2$, $\Gamma(A[X]) = I_2$. There exists a $M(X, T)\in SL_2(A[X, T])$ such that $M(X, 0) = I_2$  and $M(X, 1) = \alpha(X)$. Clearly, $\chi([M(0, T)], [M(1, T)]) = (\alpha(X), I_2)$. This shows that $\chi$ is surjective and hence, $\frac{\pi_1(SL_2(A\oplus A))}{\pi_1(SL_2(A))}\cong \Gamma(B).$ This implies that $\pi_1(SL_2(A))\cong \Gamma(B).$
\end{proof}

The following theorem is an version of Mayer Vietoris sequence.
\begin{theorem}\label{fiber direct} For the fibre product diagram $\begin{tikzcd}
B \arrow{r}{} \arrow[swap]{d}{} & A[X] \arrow{d}{\delta} \\
A\arrow{r}{\Delta} & A\oplus A
\end{tikzcd}$, we have the following exact sequence \\$\pi_1(SL_2(B)) \overset{\Psi_1}\to\pi_1(SL_2(A[X]))\oplus \pi_1(SL_2(A)) \overset{\Psi_2}\to\pi_1(SL_2(A\oplus A)) \overset{\chi}\to \Gamma(B)\overset{\Phi_1}\to \Gamma(A[X])\oplus \Gamma(A)\overset{\Phi_2}\to \Gamma(A\oplus A)$, where $$\Psi_1([\alpha(X)(T), \beta(T)])=([\alpha(X)(T)], [\beta(T)]),$$ $$\Psi_2([\alpha(X)(T)], [\beta(T)])=([\alpha(0)(T)\beta(T)^{-1}], [\alpha(1)(T)\beta(T)^{-1}]),$$ $$\Phi_1([\alpha(X), \beta])=([\alpha(X)], [\beta]),$$ $$\Phi_2([\alpha(X)], [\beta])=([\alpha(0)\beta^{-1}], [\alpha(1)\beta^{-1}]).$$
\end{theorem}
\begin{proof}
Since $[\alpha(0)(T)] =  [\alpha(1)(T)]$ in $\pi_1(SL_2(A[X]))$, $[\alpha(0)(T)\beta(T)^{-1}] =  [\alpha(1)(T)\\\beta(T)^{-1}]$. Now by Theorem \ref{direct sum}, we have $\ker(\chi)= Im(\Psi_2)$.

\noindent\textbf{Claim:  $\ker(\Psi_2)= Im(\Psi_1)$}

It is clear that  $Im(\Psi_1)\subseteq \ker(\Psi_2)$. Conversely, suppose $([\alpha(X)(T)], [\beta(T)])\in \ker(\Psi_2)$. Then, $([\alpha(0)(T)\beta(T)^{-1}], [\alpha(1)(T)\beta(T)^{-1}]) = ([I_2], [I_2])$ in $\pi_1(SL_2(A\oplus A))$. Hence, there exists $(\gamma(T, S), \gamma'(T, S)) \in SL_2(A\oplus A)[T, S]$ such that $$(\gamma(T, 0), \gamma'(T, 0)) = (I_2, I_2)$$ $$(\gamma(T, 1), \gamma'(T, 1))= (\alpha(0)(T)\beta(T)^{-1}, \alpha(1)(T)\beta(T)^{-1})$$ and $$(\gamma(0, S), \gamma'(0, S))= (\gamma(1, S), \gamma'(1, S))=(I_2, I_2).$$

Consider, $\theta(X, T) = \gamma(T, 1-X)^{-1} \gamma'(T, X)^{-1}\alpha(X)(T)$. Then, it is clear that $\theta(0, T) = \theta(1, T) = \beta(T)$. Hence, $(\theta(X, T), \beta(T))\in \pi_1(SL_2(B))$.

Now take  $(M(X, T, S), \beta(T))  = (\gamma(T, (1-X)S)^{-1} \gamma'(T, XS)^{-1}\alpha(X)(T), \beta(T))$. Then it is clear that $$(M(X, T, 0), \beta(T))  = (\alpha(X)(T), \beta(T)),$$ $$(M(X, T, 1), \beta(T))  = (\theta(X, T), \beta(T)),$$ $$(M(X, 0, S), \beta(0))  =(M(X, 1, S), \beta(1))  =(I_2, I_2).$$ 

Thus, $[\theta(X, T), \beta(T)] = [\alpha(X)(T), \beta(T)]$ in $\pi_1(SL_2(A[X]))\oplus \pi_1(SL_2(A))$. So, $\Psi_1([\theta(X, T), \beta(T)]) = [\alpha(X)(T), \beta(T)]$. This shows that $\ker(\Psi_2)\subseteq Im(\Psi_1)$. Therefore, $\ker(\Psi_2)= Im(\Psi_1)$.

\noindent\textbf{Claim:  $\ker(\Phi_1)= Im(\chi)$}

By the definition  of $\chi$, $Im(\chi)\subseteq \ker(\Phi_1).$ Conversely, suppose $\Phi_1([\alpha(X), \beta])=([I_2, I_2]).$ Then, there exists $(\theta(X, T), \sigma(T))\in SL_2(A[X, T])\times SL_2(A[T])$ such that 
$(\theta(X, 0), \sigma(0)) = (I_2, I_2)$ and $(\theta(X, 1), \sigma(1)) = (\alpha(X), \beta)$.

For $(\sigma(T)^{-1}, \sigma(T)^{-1})$, we have $(\sigma(0)^{-1}, \sigma(0)^{-1}) = (I_2, I_2)$ and 
$(\sigma(1)^{-1}, \sigma(1)^{-1})\\(\alpha(X), \beta) = (\beta^{-1}\alpha(X), I_2).$ This shows that $[\alpha(X), \beta] = [\beta^{-1}\alpha(X), I_2]$ in $\Gamma(B)$. Consider, $M(X, T) = \sigma(T)^{-1}\theta(X, T)$. Then $[M(0, T), M(1, T)] \in \pi_1(SL_2(A\oplus A))$ and $\chi([M(0, T), M(1, T)]) = (\beta^{-1}\alpha(X), I_2).$ Hence, $\ker(\Phi_1)\subseteq Im(\chi).$ So, $\ker(\Phi_1)= Im(\chi).$

\noindent\textbf{Claim:  $\ker(\Phi_2)= Im(\Phi_1)$}

It is clear that  $Im(\Phi_1)\subseteq \ker(\Phi_2)$. Conversely, suppose $([\alpha(X), \beta])\in \ker(\Phi_2)$. This implies that $[\alpha(0)\beta^{-1}, \alpha(1)\beta^{-1}] = [I_2, I_2]$ in $\Gamma(A\oplus A)$. So, there exists $(\sigma(T), \sigma'(T)) \in SL_2((A\oplus A)[T])$ such that $$(\sigma(0), \sigma'(0))= (I_2, I_2)$$ and $$(\sigma(1), \sigma'(1))(\alpha(0)\beta^{-1}, \alpha(1)\beta^{-1})= (I_2, I_2) \Rightarrow (\sigma(1), \sigma'(1))(\alpha(0), \alpha(1))= (\beta, \beta).$$

Consider $\gamma(X) = \sigma'(X)\sigma(1-X)\alpha(X)$. Then, $[\gamma(X), \beta] \in \Gamma(B)$.

Now, take $(\theta(X, T), I_2) = (\sigma'(XT)\sigma((1-X)T), I_2) \in SL_2((A[X]\oplus A)[T])$. Then $(\theta(X, 0), I_2) = (I_2, I_2)$ and $(\theta(X, 1), I_2)(\alpha(X), \beta) = (\gamma(X), \beta)$. This shows that $[\gamma(X), \beta] = [\alpha(X), \beta]$ in  $\Gamma(A[X])\oplus \Gamma(A)$. This shows that $\ker(\Phi_2)\subseteq Im(\Phi_1)$. Therefore, $\ker(\Phi_2)= Im(\Phi_1)$.
\end{proof}

Now, by using Corollary \ref{trivial}, we prove an algebraic counterpart of $H^1(S^1, \mathbb{Z})\cong \mathbb{Z}$.

\subsection{Algebraic counterpart of $H^1(S^1, \mathbb{Z})\cong \mathbb{Z}$:}

We know that the coordinate ring $\frac{\mathbb{R}[X, Y]}{(X^2+Y^2-1)}$ of real circle $S^1$ is the ring of polynomial functions from $S^1 \to \mathbb{R}$. Since a unit circle is homeomorphic to the quotient of unit interval by identifying  $0$ and $1$ as a single point, the ring $\frac{\mathbb{R}[X, Y]}{(X^2+Y^2-1)}$ is similar to the ring of those polynomial functions from $[0, 1] \to \mathbb{R}$ which have the same image at  $0$ and $1$. Therefore, we can think of $\frac{\mathbb{R}[X, Y]}{(X^2+Y^2-1)}$ to be like $B$, where $B = \{f(X)\in \mathbb{R}[X] \mid f(0) = f(1)\}$. In fact, $B$ is the fibre product of $\mathbb{R}[X]$ and  $\mathbb{R}$ over $\mathbb{R}\oplus \mathbb{R}$ by the given ring homomorphisms $\Delta: \mathbb{R} \to \mathbb{R}\oplus \mathbb{R}$ defined by $\Delta(a)= (a, a)$, and $\delta: \mathbb{R}[X] \to \mathbb{R}\oplus \mathbb{R}$ defined by $\delta(f(X))= (f(0), f(1))$, that is, we have the following fibre product diagram \[
\begin{tikzcd}
B \arrow{r}{} \arrow[swap]{d}{} & \mathbb{R}[X] \arrow{d}{\delta} \\
\mathbb{R}\arrow{r}{\Delta} & \mathbb{R}\oplus \mathbb{R}
\end{tikzcd}
\] 
Since $\Gamma(\mathbb{R})$ is trivial, by this fibre product diagram and Corollary \ref{trivial}, we have $$\Gamma(B)\cong\pi_1(SL_2(\mathbb{R})).$$ 

\subsection{Algebraic counterpart of $H^1(S^2, \mathbb{Z})\cong \{0\}$:} 
In topology, we know that $H^1(S^2, \mathbb{Z})\cong \{0\}$, that is, any continuous map from $S^2 \to S^1$ is homotopic to a constant. As an algebraic counterpart, we prove that $\Gamma(B) \cong \{I_2\}$ by assuming the algebraic counterpart of the fact that $S^1$ is connected, where $B$ is a ring similar to the coordinate ring of the real 2-sphere.

Since the 2-sphere is homeomorphic to a closed disc $D^2$ in $\mathbb{R}^2$ with its boundary identified to a single point and the coordinate ring of the 2-sphere is the ring of polynomial functions from $S^2 \to \mathbb{R}$, we can say that the coordinate ring of the 2-sphere is like the ring of polynomial functions from $D^2 \to \mathbb{R}$ which are constant on $S^1$. Thus, $B = \{f(X, Y)\in \mathbb{R}[X, Y]  \mid \overline{f(X, Y)} = \bar{a}\}$ is like the coordinate ring of the real 2-sphere (where bar denotes the modulo ${(X^2+Y^2-1)}$ and $a \in \mathbb{R}$). In fact, $B$ is the fibre product of $\mathbb{R}[X, Y]$ and  $\mathbb{R}$ using ring homomorphisms $i: \mathbb{R} \to\frac{\mathbb{R}[X, Y]}{(X^2+Y^2-1)}$ defined by $i(a)= \bar{a}$, and $\delta: \mathbb{R}[X, Y] \to \frac{\mathbb{R}[X, Y]}{(X^2+Y^2-1)}$ defined by $\delta(f(X, Y))= \overline{f(X, Y)}$, that is, we have the following commutative diagram \[\begin{tikzcd}
B \arrow{r}{} \arrow[swap]{d}{} & \mathbb{R}[X, Y] \arrow{d}{\delta} \\
\mathbb{R} \arrow{r}{i} & \frac{\mathbb{R}[X, Y]}{(X^2+Y^2-1)}
\end{tikzcd}
\]   
%Hence, $SL_2(B) = \{(\alpha(X, Y), \beta)\in SL_2(\mathbb{R}[X, Y])\times SL_2(\mathbb{R}) \mid \overline{\alpha(X, Y)} = \overline{\beta}\}$. 
%It is easy to verify that we have the following exact sequence $$\pi_1(SL_2(B)) \overset{\psi_1}\to\pi_1(SL_2(\mathbb{R}[X, Y]) \oplus \pi_1(SL_2(\mathbb{R})) \overset{\psi_2}\to\pi_1\bigg(SL_2\bigg(\frac{\mathbb{R}[X, Y]}{(X^2+Y^2-1)}\bigg)\bigg)$$
%
%where $$\psi_1([\alpha(X, Y)(T), \beta(T)])=[\alpha(X, Y)(T), \beta(T)]$$
%
%$$\psi_2([\alpha(X, Y)(T)], [\beta(T)])=[\overline{\alpha(X, Y)(T)\beta(T)^{-1}}]$$

\begin{remark}\label{surjective}
 $\pi_1(SL_2(\mathbb{R}[X, Y])\cong \pi_1(SL_2(\mathbb{R}))$ (\cite{RSS2}).
\end{remark}

\begin{theorem}
Assume that $\pi_1\bigg(SL_2\bigg(\frac{\mathbb{R}[X, Y]}{(X^2+Y^2-1)}\bigg)\bigg) \cong \pi_1(SL_2(\mathbb{R}))$. Then $\Gamma(B) = \{I_2\}$.
\end{theorem}
\begin{proof}
Let $(\alpha(X, Y), \beta)\in \Gamma(B)$. Then $\alpha(X, Y)\in \Gamma(\mathbb{R}[X, Y])$ and $\beta\in \Gamma(\mathbb{R})$. Since $\Gamma(\mathbb{R}[X, Y])$ and $\Gamma(\mathbb{R})$ are trivial, there exist $\theta(X, Y)(T) \in \mathbb{R}[X, Y] [T]$ and $\gamma(T) \in \mathbb{R}[T]$ such that $\theta(X, Y)(0) = I_2, \theta(X, Y)(1) = \alpha(X, Y)$ and $\gamma(0) = I_2, \gamma(1) = \beta$.

Consider $M(X, Y)(T) =\theta(X, Y)(T)\gamma(T)^{-1}$. Then it is clear that $\overline{M(X, Y)(0)}= \overline{M(X, Y)(1)} = I_2.$ Thus $\overline{M(X, Y)(T)}\in \pi_1\bigg(SL_2\bigg(\frac{\mathbb{R}[X, Y]}{(X^2+Y^2-1)}\bigg)\bigg)$

By assumption and Remark \ref{surjective}, there exists $\big([\eta(X, Y)(T)], [\delta(T)]\big) \in \pi_1(SL_2(\mathbb{R}[X, Y])\\ \oplus \pi_1(SL_2(\mathbb{R}))$ such that $\overline{M(X, Y)(T)} = \overline{\eta(X, Y)(T)\delta(T)^{-1}}$. 

Clearly, the matrix $\big(\eta(X, Y)(T)^{-1}\theta(X, Y)(T), \delta(T)^{-1}\gamma(T)\big) \in SL_2(B[T])$. Also $$\big(\eta(X, Y)(0)^{-1}\theta(X, Y)(0), \delta(0)^{-1}\gamma(0)\big) = (I_2, I_2)$$ and $$\big(\eta(X, Y)(1)^{-1}\theta(X, Y)(1), \delta(1)^{-1}\gamma(1)\big) = (\alpha(X, Y), \beta).$$ Therefore, $\Gamma(B) = \{I_2\}$.
\end{proof}

\begin{remark}
The assumption that $\pi_1\bigg(SL_2\bigg(\frac{\mathbb{R}[X, Y]}{(X^2+Y^2-1)}\bigg)\bigg) \cong \pi_1(SL_2(\mathbb{R}))$ is the algebraic assumption corresponding to the topological fact the $S^1$ is path connected which is needed to show that $S^2$ is simply connected and any map $S^2 \to S^1$ is homotopic to a constant. We conjecture that this assumption is valid.
\end{remark}

\section{On some heuristic examples and remarks}
Let   \[\begin{tikzcd}
A \arrow{r}{} \arrow[swap]{d}{} & B \arrow{d}{f} \\
C \arrow{r}{g} & D
\end{tikzcd}
\] be a fibre product diagram of rings corresponding to a Milnor's square. Here, we assume that $f$ is a surjective ring homomorphism. Clearly, $A= \{(b, c) \mid f(b) = g(c)\}$. Let $((b_1, c_1), (b_2, c_2), (b_3, c_3))\in Um_3(A)$ such that $(b_1, b_2, b_3)$ and $(c_1, c_2, c_3)$ are completable over $B$ and $C$, respectively. , we will get a co-cycle  in $SL_2(D)$ (as in \cite{RS}). We consider the image of this co-cycle in $\Gamma(D)$. We expect that the unimodular row to be completable if the associated co-cycle belongs to $im(\Gamma(f))im(\Gamma(g))$

Since $(b_1, b_2, b_3)$ and $(c_1, c_2, c_3)$ are completable over $B$ and $C$, there exist $\theta \in SL_3(B)$ and $\sigma\in SL_3(C)$ such that $$\theta\begin{pmatrix}
b_1\\ b_2\\ b_3
\end{pmatrix} = \begin{pmatrix}
1\\ 0\\ 0
\end{pmatrix} ~\text{and}~ \sigma\begin{pmatrix}
c_1\\ c_2\\ c_3
\end{pmatrix} = \begin{pmatrix}
1\\ 0\\ 0
\end{pmatrix}.$$ Therefore, $$g(\sigma) f(\theta)^{-1}\begin{pmatrix}
1\\ 0\\ 0
\end{pmatrix} = \begin{pmatrix}
1\\ 0\\ 0
\end{pmatrix}.$$ Thus, we have a matrix $\lambda\in SL_2(D)$ and $d_{12}, d_{13} \in D$ such that $$g(\sigma) f(\theta)^{-1} = \begin{pmatrix}
1&d_{12} &d_{13}\\ 0& \lambda&\\ 0& &
\end{pmatrix}.$$

\begin{definition} The element $\lambda \in SL_2(D)$ is called a co-cycle associated to the unimodular row $((b_1, c_1), (b_2, c_2), (b_3, c_3))$. As we know that every stably free module of rank $2$ is equivalent to a unimodular row of length $3$, so we have a co-cycle associated by a stably free module of rank $2$ over $B$.
\end{definition}

\begin{proposition}\label{splits}
Let $((b_1, c_1), (b_2, c_2), (b_3, c_3))\in Um_3(A)$  such that $(b_1, b_2, b_3)$ and $(c_1, c_2, c_3)$ are completable over $B$ and $C$, respectively. Let $\lambda \in SL_2(D)$ be an associated co-cycle. Then $((b_1, c_1), (b_2, c_2), (b_3, c_3))$ is completable over $A$ if and only if $\lambda$ splits.
\end{proposition}
\begin{proof} 
Since $\lambda$ is an associated co-cycle, $g(\sigma) f(\theta)^{-1} = \begin{pmatrix}
1&d_{12} &d_{13}\\ 0& \lambda&\\ 0& &
\end{pmatrix},$ where $\theta \in SL_3(B)$ and $\sigma\in SL_3(C)$ are completion of $(b_1, b_2, b_3)$ and $(c_1, c_2, c_3)$, respectively, and $d_{12}, d_{13}\in D$.

Suppose $((b_1, c_1), (b_2, c_2), (b_3, c_3))$ is completable. Then there exists a matrix $(M, N)\in SL_3(A)$ such that $$M\begin{pmatrix}
b_1\\ b_2\\ b_3
\end{pmatrix} = \begin{pmatrix}
1\\ 0\\ 0
\end{pmatrix} ~\text{and}~ N\begin{pmatrix}
c_1\\ c_2\\ c_3
\end{pmatrix} = \begin{pmatrix}
1\\ 0\\ 0
\end{pmatrix}.$$

Clearly,  $$f(M)f(\theta)^{-1}\begin{pmatrix}
1\\ 0\\ 0
\end{pmatrix} = \begin{pmatrix}
1\\ 0\\ 0
\end{pmatrix} ~\text{and}~ g(N)g(\sigma)^{-1}\begin{pmatrix}
1\\ 0\\ 0
\end{pmatrix} = \begin{pmatrix}
1\\ 0\\ 0
\end{pmatrix}.$$ Therefore, $$f(M)f(\theta)^{-1} = \begin{pmatrix}
1&\star &\star\\ 0& f(\gamma)&\\ 0& &
\end{pmatrix}~\text{and}~ g(N)g(\sigma)^{-1} = \begin{pmatrix}
1&\star &\star\\ 0& g(\tau)&\\ 0& &
\end{pmatrix},$$ for some $\gamma \in SL_2(B)$ and $\tau\in SL_2(C).$

Since $f(M) = g(N)$, $g(\sigma) f(\theta)^{-1} = (g(N)g(\sigma)^{-1})^{-1}f(M)f(\theta)^{-1} $. Thus,  $$\begin{pmatrix}
1&a_{12} &a_{13}\\ 0& \lambda&\\ 0& &
\end{pmatrix}= \begin{pmatrix}
1&\star &\star\\ 0& g(\tau)^{-1}&\\ 0& &
\end{pmatrix}\begin{pmatrix}
1&\star &\star\\ 0& f(\gamma)&\\ 0& &
\end{pmatrix}.$$ Hence, $\lambda = g(\tau)^{-1}f(\gamma)$, that is, $\lambda$ splits. 

Conversely, suppose the co-cycle $\lambda$ splits, that is, there exist $\gamma\in SL_2(B)$ and $\delta\in SL_2(C)$ such that $$\lambda = g(\delta)f(\gamma).$$

Therefore, $$g(\sigma) f(\theta)^{-1} = \begin{pmatrix}
1&0 &0\\ 0& g(\delta)&\\ 0& &
\end{pmatrix}\begin{pmatrix}
1&d_{12} &d_{13}\\ 0& f(\gamma)&\\ 0& &
\end{pmatrix}.$$
 Since $f$ is surjective, there exists $b_{12}$ and $b_{13}$ in $B$ such that $f(b_{12}) = d_{12}$ and $f(b_{13}) = d_{13}$

Consider  $M = \begin{pmatrix}
1&b_{12} &b_{13}\\ 0& \gamma&\\ 0& &
\end{pmatrix}\theta \in SL_3(B)$
and
$N = \begin{pmatrix}
1&0 &0\\ 0& \delta^{-1}&\\ 0& &
\end{pmatrix}\sigma \in SL_3(C)$. It is clear that $f(M) = g(N),$ $M\begin{pmatrix}
b_1\\ b_2\\ b_3
\end{pmatrix} = \begin{pmatrix}
1\\ 0\\ 0
\end{pmatrix}$ and $N\begin{pmatrix}
c_1\\ c_2\\ c_3
\end{pmatrix} = \begin{pmatrix}
1\\ 0\\ 0
\end{pmatrix}$. Thus, $(M^{-1}, N^{-1}) \in SL_3(A)$ is a completion of the unimodular row $((b_1, c_1), (b_2, c_2), (b_3, c_3))$. 
\end{proof}
\begin{remark}\label{obs} We see that every unimodular row of length $3$ over $A$ gives a co-cycle in $SL_2(D)$ which is stably elementary.  We consider the associated co-cycle to a unimodular row over $A$ as an element in $\Gamma(D)$. By Proposition \ref{splits}, we also expect that the unimodular row to be completable if the associated co-cycle belongs to $im(\Gamma(f))im(\Gamma(g))$. That is, the obstruction for a unimodular row of length $3$ to be completable should lie in the group $\frac{\Gamma(D)}{im(\Gamma(f))im(\Gamma(g))}.$ By Milnor's construction of projective modules (see \cite{JM}), every stably elementary element of $SL_2(D)$ gives a stably free module over $A$ of rank $2$, that is, an unimodular row of length $3$ over $A$. Thus, the obstruction for a freeness of a stably free module of rank $2$ lies in the group $\frac{\Gamma(D)}{im(\Gamma(f))im(\Gamma(g))}.$ Now, we explore some heuristic examples.
\end{remark}

\vspace{0.2cm}
\noindent\textbf{Stably free modules over cylinder:} One can get a cylinder by identifying points on two parallel lines in plane. By this motivation, we look at $\mathbb{R}^2$, and two lines $x = 0$ and $x =1$.

Consider $A = \mathbb{R}[Y]$ and the ring homomorphisms 
$\Delta: \mathbb{R}[Y] \to\mathbb{R}[Y]\oplus\mathbb{R}[Y]$ defined as $\Delta(f(Y))= (f(Y), f(Y))$, and $\delta: \mathbb{R}[X, Y] \to \mathbb{R}[Y]\oplus\mathbb{R}[Y]$ defined as $\delta(f(X, Y))= (f(0, Y), f(1, Y))$. Thus, we have the following fibre product diagram \[\begin{tikzcd}
B \arrow{r}{} \arrow[swap]{d}{} & \mathbb{R}[X, Y] \arrow{d}{\delta} \\
\mathbb{R}[Y] \arrow{r}{\Delta} & \mathbb{R}[Y]\oplus\mathbb{R}[Y]
\end{tikzcd}
\] where $B = \{(f(X, Y), g(Y)) \mid f(0, Y) = f(1, Y)= g(Y)\}$. The ring $B$ is an algebraic counterpart of the ring of continuous functions on the cylinder. 

Since $SL_2(\mathbb{R}[Y]) = E_2(\mathbb{R}[Y])$, it is easy to see that an associated cocycle in $SL_2(\mathbb{R}[Y]\oplus\mathbb{R}[Y])$ to a unimodular row over $B$ always splits. So, one expects that any unimodular row of length $3$ over $B$ is completable.

\vspace{0.2cm}

\noindent\textbf{Stably free modules over torus:} We know that a torus is the cartesian  product of $S^1 \times S^1$ and a unit circle is homeomorphic to the quotient of unit interval by identifying  $0$ and $1$ as a single point. Hence, a torus is the quotient of the cylinder $S^1 \times [0, 1]$ by identifying  $(x, 0)$ and $(x, 1)$ as a single point for every $x\in S^1$. So, we can say that the coordinate ring of torus is the ring of polynomial functions on $S^1 \times [0, 1]$ to $\mathbb{R}$ whose images are same at $(x, 0)$ and $(x, 1)$ for every $x\in S^1$.

Let $A = \frac{\mathbb{R}[X, Y]}{(X^2+Y^2-1)}$ in the fibre product diagram (\ref{1}).
Then, $B = \{f(T)\in A[T]\times A \mid f(0) = f(1)\}$. We can say $B$ is like the coordinate ring of torus. By Remark \ref{obs}, one expects that the obstruction to the freeness of stably free module of rank 2 over the coordinate ring of torus lies in the group $\frac{\Gamma(A\oplus A)}{im(\Gamma(\Delta))im(\Gamma(\delta))}\cong\Gamma(A)$.

\vspace{0.2cm}

\noindent\textbf{Stably free modules over Klein bottle:} 
Consider the ring $A = \frac{\mathbb{R}[X, Y]}{(X^2+Y^2-1)}$. Let $h: A \longrightarrow A$ be the ring homomorphism defined by $h(x) =x$ and $h(y) = -y$, where $x$ and $y$ are the images of $X$ and $Y$ in the ring $A$. Consider the ring homomorphisms $\Delta: A \to A\oplus A$ defined by $\Delta(a)= (a, h(a))$, and $\delta: A[T] \to A\oplus A$ defined by $\delta(f(T))= (f(0), f(1))$. Suppose $C$ is the pull back of the diagram \[\begin{tikzcd}
C \arrow{r}{} \arrow[swap]{d}{} & A[T] \arrow{d}{\delta} \\A \arrow{r}{\Delta} & A\oplus A
\end{tikzcd}
\]

Then, $C = \{(f(T), a)\in A[T]\times A \mid f(0) = a,  f(1)= h(a)\}$ and it is like the coordinate ring of Klein bottle. It is easy to see that  if $\Gamma(A) \cong \mathbb{Z}$, then $im(\Gamma(\delta))$ is generated by $(1, 1)$ and $im(\Gamma(\Delta))$ is generated by $(1, -1)$. Hence, $\frac{\Gamma(A\oplus A)}{im(\Gamma(\Delta))im(\Gamma(\delta))}\cong \mathbb{Z}_2$. By Remark \ref{obs}, one expects that  the obstruction for the freeness of stably free modules of rank $2$ over the coordinate ring of Klein bottle lies in the group $\mathbb{Z}_2$. 
\vspace{0.2cm}

\noindent\textbf{Stably free modules over projective space:} We know that the real projective space $\mathbb{P}^2(\mathbb{R})$ is equivalent to unit disk with antipodal points on the  boundary identified. In fact, set of all continuous functions on $\mathbb{P}^2(\mathbb{R})$ is equal to the set of all continuous functions $f$ on $\mathbb{R}^2$ whose restriction to $S^1$ satisfy  $f(-x, -y) = f(x, y)$.

Now, consider the pull back diagram 
\[\begin{tikzcd}
P \arrow{r}{} \arrow[swap]{d}{} & \mathbb{R}[X, Y] \arrow{d}{\nu} \\
\frac{\mathbb{R}[X, Y]}{(X^2+Y^2-1)} \arrow{r}{\mu} & \frac{\mathbb{R}[X, Y]}{(X^2+Y^2-1)}
\end{tikzcd}
\]
where $\nu$ is the natural quotient map and $\mu(f(x, y))= f(x^2-y^2, 2xy)$. So, $P$ is like the ring of polynomial functions on $\mathbb{P}^2(\mathbb{R})$. Since $\Gamma(\mathbb{R}[X, Y])$ is a trivial group, as in Remark \ref{obs}, the obstruction for a unimodular row to be completable is in the group $\frac{\Gamma(A)}{im(\Gamma(\mu))}$, where $A = \frac{\mathbb{R}[X, Y]}{(X^2+Y^2-1)}$ (that is, $\frac{\Gamma(A)}{im(\Gamma(\mu))}$ is like  $\mathbb{Z}_2$).

\vspace{0.2cm}

\textbf{On Swan's example:} In \cite{SW}, Swan gave examples to answer  in the affirmation the following question of Murthy and Wiegand:
``Does there exist a commutative ring $R$ and a finitely generated projective module $L$ over $R$ of rank $1$ such that $L\oplus L^{-1}$ is stably free but not free?

%By using one of the examples given by Swan, we see that there exists a ring $B$ such that projective modules corresponding  to some elements of $G(B)$ are stably free but not free. 

To understand one of the example given by Swan, consider the pull back diagram 
\[\begin{tikzcd}
B \arrow{r}{} \arrow[swap]{d}{} & \mathbb{R}[X, Y] \arrow{d}{\nu} \\
\frac{\mathbb{R}[X, Y]}{(X^2+Y^2-1)} \arrow{r}{\mu} & \frac{\mathbb{R}[X, Y]}{(X^2+Y^2-1)}
\end{tikzcd}
\]
where $\nu$ is the natural quotient map and $\mu: \frac{\mathbb{R}[X, Y]}{(X^2+Y^2-1)} \to \frac{\mathbb{R}[X, Y]}{(X^2+Y^2-1)}$ is the homomorphism induced by the map $S^1 \to S^1$ of degree $n$ given by $z \to z^n$, where $n> 2$. By Remark \ref{obs}, any stably free module of rank 2 over $B$ is free if the associated co-cycle in $SL_2(\frac{\mathbb{R}[X, Y]}{(X^2+Y^2-1)})$ splits.

Let $\tau = \begin{pmatrix}
x & y\\
-y & x
\end{pmatrix}$, where $x , y $ denote the image of $X, Y$ in $\frac{\mathbb{R}[X, Y]}{(X^2+Y^2-1)}.$ Consider $\sigma = \tau^2.$ Then $\begin{pmatrix}
1&\star &\star\\ 0& \sigma &\\ 0& &
\end{pmatrix}$ is stably elementary. Since $\sigma$ does not splits for $n>2$, the corresponding stably free module of rank 2 over $B$ is not free. Thus, we have a non-trivial elements of $\Gamma(\frac{\mathbb{R}[X, Y]}{(X^2+Y^2-1)})$ provides a non free stably free module of rank $2$ over $B$.

\end{document}